\documentclass[11pt]{article}
\usepackage{fullpage}
\usepackage{amsmath,amssymb,amsfonts,amsthm}
\usepackage{epsfig}
\usepackage{enumitem}
\usepackage{xcolor}
\usepackage{tikz, ifthen}
\usepackage[hidelinks]{hyperref}
\usepackage{authblk}

\usepackage{accents}

\usetikzlibrary{quotes, positioning, arrows.meta}

\newtheorem{theorem}{Theorem}[section]
\newtheorem{lemma}[theorem]{Lemma}

\numberwithin{equation}{section}

\newcommand{\head}[1]{{\sc #1}}

\newcommand{\jou}[4]{{\it #1} {\bf #2} (#3), #4.}

\newcommand{\DF}[1]{{\bf #1\/}}

\newcommand{\Pe}{\mathbb{P}}

\newcommand{\fett}[1]{\mathbf{#1}}

\newcommand{\mdo}[3]{#1 \equiv #2 \, \mathrm{(mod} \; #3 \mathrm{)}}

\newcommand{\set}[2]{\{#1 \;|\; #2 \}}
\newcommand{\ems}{\varnothing}
\newcommand{\sm}{\setminus}

\newcommand{\overbar}[1]{\mkern 1.5mu\overline{\mkern-1.5mu#1\mkern-1.5mu}\mkern 1.5mu}

\hypersetup{
	colorlinks,
	linkcolor={red!50!black},
	citecolor={blue!50!black},
	urlcolor={blue!80!black}
}

\begin{document}

\title{\bf The Tournament Theorem of R\'{e}dei revisited}

\author[,1]{Thomas Schweser\thanks{Email: {\tt thomas.schweser@th-rosenheim.de}}}
\author[,2]{Michael Stiebitz\thanks{Email: {\tt michael.stiebitz@tu-ilmenau.de}}}
\author[,3]{Bjarne Toft\thanks{Email: {\tt btoft@imada.sdu.dk}}}

\affil[1]{Technische Hochschule Rosenheim, Rosenheim, Germany}
\affil[2]{Technische Universit\"at Ilmenau, Ilmenau, Germany}
\affil[3]{University of Southern Denmark, Odense, Denmark}

\date{}
\maketitle

\begin{abstract}
In 1934 L. R\'{e}dei published his famous theorem that the number of Hamiltonian paths in a tournament is odd. In fact it is a corollary of a stronger theorem in his paper. Stronger theorems were also obtained in the early 1970s by G.A. Dirac in his lectures at Aarhus University and by C. Berge in his monographs on graphs and hypergraphs. We exhibit the stronger theorems of R\'{e}dei, Dirac, and Berge and we explain connections between them. The stronger theorem of Dirac has two corollaries, one equivalent to R\'{e}dei's stronger theorem and the other related to Berge's stronger theorem.
\end{abstract}

\noindent{\small{\bf AMS Subject Classification:} 05C20, 05C45}

\noindent{\small{\bf Keywords:} Digraphs, Tournaments, Hamiltonian Path}

\section{Introduction}

When Beineke, Toft, and Wilson were working on \textit{Milestones in Graph Theory} \cite{Milestones}, they noticed that the 1934 paper by L. R\'{e}dei \cite{Redei} contains a stronger theorem than \textbf{R\'{e}dei's Theorem} that a tournament always has an odd number of Hamiltonian paths. The stronger theorem deals with mixed graphs. A \DF{mixed graph} $G$ is obtained from a (simple) graph $G'$ by orienting some of the edges of $G'$ from one of its ends to the other. A \DF{Hamiltonian path} in $G$ is a sequence of vertices of $G$ that contains each vertex of $G$ exactly once and in which each pair of consecutive vertices in the sequence is joined either by a non-oriented edge of $G$ or by an edge of $G$ oriented in the forward direction of the sequence. Note that a Hamiltonian path is not a subgraph but a sequence of vertices. An edge joining two consecutive vertices of the sequence is said to be \DF{contained} in the Hamiltonian path. If all edges contained in a Hamiltonian path are non-oriented, then the reverse sequence is a different Hamiltonian path.

\begin{theorem} \head{(R\'{e}dei's Stronger Theorem)} Let $T$ be a tournament with at least $2$ vertices. Add to $T$ a new non-empty set $W$ of vertices and some new non-oriented edges between $W$ and $T$ and between vertices of $W$.  Then the number of Hamiltonian paths in the new mixed graph $G$, beginning and ending in $T$, is even.
\label{theorem:redeistrong}
\end{theorem}

Note that $0$ is an even number! A Hamiltonian path in the mixed graph $G$ in Theorem~\ref{theorem:redeistrong} contains a non-empty set of non-oriented edges and a set (perhaps empty) of oriented edges. 

\smallskip

Let $T$ be a tournament and let $(u,v)$ be an oriented edge of $T$. If we replace $(u,v)$ by $(v,u)$ we obtain a new tournament $T'$. Let $G$ be the mixed graph obtained from $T$ by joining a new vertex to $u$ and $v$ by two non-oriented edges. By R\'{e}dei's Stronger Theorem, the number of Hamiltonian paths in $G$, beginning and ending in $T$, is even. This implies that the number of Hamiltonian paths in $T$ and $T'$ containing the oriented edge between $u$ and $v$ are of the same parity. Consequently, the parity of the number of Hamiltonian paths in $T$ and $T'$ are the same, too. Since the transitive tournament has exactly one Hamiltonian path, it follows that $T$ has an odd number of Hamiltonian paths.

\smallskip

In \cite{Berge} and in its later editions and translations, C. Berge deduced R\'{e}dei's Theorem from a stronger result. 

\begin{theorem} \head{(Berge's Stronger Theorem)} Let $G$ be a mixed graph with at least $2$ vertices and let $\overbar{G}$ be its complement (a non-oriented edge and a non-edge with the same two ends are complements; an oriented edge $(u,v)$ and its reverse-oriented edge $(v,u)$ are complements). Then the numbers of Hamiltonian paths in $G$ and $\overbar{G}$ have the same parity - both even or both odd (both cases occur).
\label{theorem:bergestrong}
\end{theorem}

Asking Carsten Thomassen whether R\'{e}dei's Stronger Theorem is generally known, he informed us that G.~A. Dirac in lectures at Aarhus University in the early 1970s also proved stronger versions of R\'{e}dei's Theorem. Consulting Dirac's handwritten lecture notes \cite{Dirac} we have generalized his approach and proved what we shall call \textbf{Dirac's Stronger Theorem}. It has a special case equivalent to R\'{e}dei's Stronger Theorem. One attractive corollary of Dirac's Stronger Theorem is the following result.

\begin{theorem} \head{(Dirac's Corollary 1)} Let $G$ be a complete mixed graph on at least $2$ vertices (each pair of vertices is joined by either a non-oriented edge or by an oriented edge). Then the number of Hamiltonian paths containing at least one non-oriented edge is even.
\label{theorem:diraccorollary1}
\end{theorem}

In the following we shall state and prove Dirac's Stronger Theorem and obtain relations to the stronger theorems of R\'{e}dei and Berge.

\section{Dirac's Stronger Theorem}

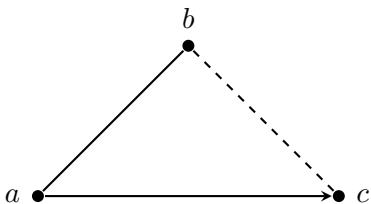
\begin{figure}[htbp]
\centering
\begin{tikzpicture}[>=stealth,thick]
  \node[circle,fill,inner sep=1.6pt,label=left:$a$]  (a) at (0,0)   {};
  \node[circle,fill,inner sep=1.6pt,label=above:$b$] (b) at (2,2) {};
  \node[circle,fill,inner sep=1.6pt,label=right:$c$] (c) at (4,0)   {};

  \draw (a) -- (b);          
  \draw[dashed] (b) -- (c);  
  \draw[->] (a) -- (c);      
\end{tikzpicture}

%
%
\caption{A mixed graph $G$ with $n=3$ vertices.}
\label{figure:mixedgraph}       
\end{figure}

In a \DF{mixed graph} $G$ each pair of vertices constitutes either a non-edge, a non-oriented edge, or an oriented edge in one of the two possible directions. We denote the set of non-edges in $G$ by $E_{1}$, the set of non-oriented edges in $G$ by $E_{2}$ and the set of oriented edges in $G$ by $E_{3}$. Moreover, by $\overbar{E_{3}}$ we denote the set $E_{3}$ with all orientations reversed. A \DF{permutation} $\fett{x}$ of $G$ is an ordering of the vertices of the mixed graph $G$ on $n$ vertices ($n \geq 2$):
$$\fett{x}=(x_{1}, x_{2}, \ldots , x_{i}, x_{i+1}, \ldots , x_{n}).$$
For the \DF{neighboring pair} $x_ix_{i+1}$ of the permutation $\fett{x}$ we have the following possibilities. Either the unordered pair 
$\{x_i,x_{i+1}\}$ belongs to $E_1 \cup E_2$, or the ordered pair $(x_i,x_{i+1})$ belongs to $E_3 \cup \overbar{E_{3}}$. In the former case, we set $e(x_ix_{i+1})=\{x_i,x_{i+1}\}$ (i.e., $e(x_ix_{i+1})$ is a non-edge or a non-oriented edge of $G$). In the latter case we set $e(x_ix_{i+1})=(x_i,x_{i+1})$ (i.e., $e(x_ix_{i+1})\in E_3$ is an oriented edge of $G$ and has the \DF{correct direction} or
$e(x_ix_{i+1})\in \overbar{E_{3}}$ is an oriented edge of $\overbar{G}$ and has the \DF{wrong direction}). We define
$E(\fett{x})=\set{e(x_ix_{i+1})}{i\in \{1,2, \ldots, n-1\}}$. An element $e\in E_1\cup E_2 \cup E_3 \cup \overbar{E_{3}}$ is \DF{contained} in $\fett{x}$ if $e\in E(\fett{x})$. Furthermore, we denote by $\Pe(G)$ the set of permutations of $G$. Note that if $G$ has $n\geq 2$ vertices, then $|E_1|+|E_2|+|E_3|=\binom{n}{2}$, $|\Pe(G)|=n!$ and $|E(\fett{x})|=n-1$ for every $\fett{x}\in \Pe(G)$. A \DF{Hamiltonian path} of $G$ is a permutation $\fett{x}\in \Pe(G)$ such that $E(\fett{x})\subseteq E_2 \cup E_3$.

\medskip

Let $G$ be the mixed graph with $n=3$ vertices depicted in Figure~\ref{figure:mixedgraph}. Then $E_1=\{\{b,c\}\}$, $E_2=\{\{a,b\}\}$, $E_3=\{(a,c)\}$, and $\overbar{E_{3}}=\{(c,a)\}$. For the permutation $\fett{x}=(c,a,b)\in \Pe(G)$, we obtain that $E(\fett{x})=\{(c,a),\{a,b\}\}=\overbar{E_{3}}\cup E_2$. If $\Pe'=\set{\fett{x}\in \Pe(G)}{\{a,b\}\in E(\fett{x})}$, then  $\Pe'=\{\fett{x}^1=(a,b,c), \fett{x}^2=(b,a,c), \fett{x}^3=(c,a,b), \fett{x}^4=(c,b,a) \}$. The only Hamiltonian path of $G$ is the permutation $\fett{x}=(b,a,c)$.

With the above notation we can now state Dirac's Stronger Theorem.

\begin{theorem} \head{(Dirac's Stronger Theorem)} Let $G$ be a mixed graph with $n \geq 2$ vertices. Let $A$ be a subset of $E_{1} \cup E_{2}$. Let $N_{A}$ denote the number of permutations of $G$ containing all elements of $A$ and no element from $\overbar{E_{3}}$. Let $N_{=A}$ denote the number of permutations of $G$ containing precisely $A$ as its elements from $E_{1} \cup E_{2}$ and containing no element from $\overbar{E_{3}}$. Then $N_{A}$ and  $N_{=A}$ have the same parity. In particular $N_{A} - N_{=A}$ is even.
\label{theorem:diracstrong}
\end{theorem}

Dirac's Corollary 1 follows from Dirac's Stronger Theorem with $A$ and $E_{1}$ both empty. To prove the theorem we use the following lemma (with the notation from Dirac's Stronger Theorem).

\begin{lemma} Let $G$ be a mixed graph with $n \geq 2$ vertices. Let $A$ be a subset of $E_{1} \cup E_{2}$ and let $D$ be a subset of $\overbar{E_{3}}$. Let  $N$ be the number of permutations of $G$ containing all elements of $A \cup D$. Then $N$ is even, except when  $|A| + |D| = n-1$, $|D|\geq 1$, and $A \cup D=E(\fett{x})$ for a permutation $\fett{x}\in \Pe(G)$, in which case $N = 1$.
\label{lemma:dirac}
\end{lemma}
\begin{proof}
Clearly, $N = 0$ except perhaps when $A \cup D$ forms the elements of disjoint sub-permutations of $G$, say  $p$ sub-permutations without any element from $D$ and $q$ sub-permutations with at least one element from $D$. In particular $|A|+|D| \leq n-1$. If we also consider each single vertex, not incident with any element of $A \cup D$, as a sub-permutation, then let the total number of sub-permutations be $r$ (in fact $r = n - |A| - |D|$). Then

$$N = r! \cdot 2^{p}.$$
It follows that $N$ is even, except when $r = 1$, $p=0$, and hence $|A|+|D| = n-1$, $|D|\geq 1$, and $A \cup D=E(\fett{x})$ for a permutation $\fett{x}\in \Pe(G)$. In this exceptional case $N = 1$. This proves the lemma.
\end{proof}

\begin{proof}[Proof of Theorem~\ref{theorem:diracstrong}]
A permutation of $G$ contains $n-1$ elements. Hence, if $|A|\geq n$, then $N_{A} = N_{=A}=0$ and if $|A|= n-1$, then $N_{A} = N_{=A}.$
So assume that $|A|\leq n-2$. We may also assume that $A$ forms the elements of $p$ disjoint sub-permutations, for otherwise we have $N_{A} = N_{=A}=0$. By Lemma~\ref{lemma:dirac}, for every set $D\subseteq \overbar{E_{3}}$, the number $M(D)$ of permutations of $G$ containing $A\cup D$ is even, except when $|D|=n-1-|A|$ and $A\cup D=E(\fett{x})$ for a permutation $\fett{x} \in \Pe(G)$, in which case $M(D)=1$. Since $N_A$ is the number of permutations of $G$ containing $A$ and no element of $\overbar{E_{3}}$, the inclusion-exclusion principle leads to
$$N_A=\sum_{D\subseteq \overbar{E_{3}}, |D|\leq n-1-|A|}(-1)^{|D|}M(D)$$
Consequently, modulo 2, $N_A$ is congruent to the number of sets $D$ of $n-1-|A|$ elements from $\overbar{E_{3}}$ that together with $A$ forms the elements of a permutation of $G$; this number is equal to $N_{=A}$ (by reversing the orders of the permutations). Hence, we have
$\mdo{N_A}{N_{=A}}{2}$.
\end{proof}

The two cases of Theorem~\ref{theorem:diracstrong}, where $A = \ems$ or $A = (E_{1} \cup E_{2}) \sm \{e\}$ for an element $e  \in  E_{1} \cup E_{2}$, imply the following two corollaries.

\begin{theorem} \head{(Dirac's Corollary 2)} Let $G$ be a mixed graph with $n \geq 2$ vertices. Let $N_\ems$ denote the number of permutations $\fett{x}\in \Pe(G)$ satisfying $E(\fett{x}) \subseteq E_{1} \cup E_{2} \cup E_{3}$. Let $N_{=\ems}$ denote the number of permutations containing no elements from $E_{1} \cup E_{2} \cup \overbar{E_3}$. Then $N_\ems$ and  $N_{=\ems}$ have the same parity. Hence $N_\ems - N_{=\ems}$ is even, i.e., the number of permutations of $G$ containing at least one element from $E_{1} \cup E_{2}$, but no element from $\overbar{E_{3}}$  is even.
\label{theorem:diraccorollary2}
\end{theorem}

\begin{theorem} \head{(Dirac's Corollary 3)} Let $G$ be a mixed graph with $n \geq 2$ vertices and with $E_{1} \cup E_{2} \neq \ems$, let $e\in E_1 \cup E_2$, and let $A = (E_{1} \cup E_{2}) \sm \{e\}$. Then $N_{A}$ and $N_{=A}$ have the same parity. Hence $N_{A} - N_{=A}$ is even, i.e. the number of permutations of $G$ containing all elements from $E_{1} \cup E_{2}$, but no element from $\overbar{E_{3}}$  is even. 
\label{theorem:diraccorollary3}
\end{theorem}

\section{R\'{e}dei's Stronger Theorem}

First, we prove R\'{e}dei's Stronger Theorem using Dirac's Corollary 3. Let $G$ be a mixed graph and $W$ be a non-empty vertex set in $G$ such that $T=G-W$ is a tournament with at least two vertices and every edge incident with a vertex in $W$ is non-oriented. Furthermore, let $\Pe'\subseteq \Pe(G)$ denote the set of Hamiltonian paths of $G$, beginning and ending in $T$. Our aim is to show that $|\Pe'|$ is even. To this end, let $E_2$ denote the set of non-oriented edges of $G$, let $E_2'$ be an arbitrary subset of $E_2$, and let $\Pe'_2=\set{\fett{x}\in \Pe'}{E(\fett{x})\cap E_2=E_2'}$. It suffices to show that $|\Pe_2'|$ is even, since
$\Pe'$ is the disjoint union of such sets. If $\Pe_2'=\ems$, this is evident. So assume that $\Pe_2'\not=\ems$. Then $E_2'\not=\ems$ and the subgraph $G_2'$ of $G$, whose edge set is $E_2'$ and whose vertex set consists of the vertices of $G$ incident with an edge in $E_2'$, is the disjoint union of paths. Each such path joins two vertices of $T$, and every vertex of $W$ belongs to exactly one path of $G_2'$. Let $T_2$ be the mixed graph obtained from the tournament $T$ by replacing an oriented edge $(u,v)$ of $T$ by a non-oriented edge $\{u,v\}$ if and only if $u$ and $v$ are joined by a subpath of $G_2'$ with all its internal vertices from $W$. Let $\Pe_2$ denote the set of Hamiltonian paths in $T_2$ that contain all of its non-oriented edges. Since $T_2$ is a complete mixed graph,  Dirac's Corollary 3 implies that $|\Pe_2|$ is even. It is easy to see that there is a one-to-one correspondence between the permutations of $\Pe_2'$ and $\Pe_2$, so $|\Pe_2'|$ is even, too, as required. This proves R\'{e}dei's Stronger Theorem. 

Consider conversely the situation in Dirac's Corollary 3. For each element $\{u,v\}$ from $E_{1} \cup E_{2}$, replace $\{u,v\}$ by either $(u,v)$ or $(v,u)$ and add an extra new vertex joined to $u$ and $v$ by two non-oriented edges. Let the set of new vertices be $W$. Then R\'{e}dei's Stronger Theorem implies Dirac's Corollary 3.

The above shows that R\'{e}dei's Stronger Theorem is equivalent with Dirac's Corollary 3 (in fact with its case of a complete mixed graph).

\section{Berge's Stronger Theorem}
 
Let $G$ be a mixed graph as before. Moreover, let $\Pe_{0}$ denote the permutations without elements from $\overbar{E_{3}}$, let $\Pe_{i}$ the permutations without elements from $E_{i} \cup \overbar{E_{3}}$ for $i = 1$ and $2$, and let $\Pe_{3}$ the permutations without elements from $E_{1} \cup E_{2} \cup \overbar{E_{3}}$ (in other words $\Pe_{3} = \Pe_{1} \cap \Pe_{2})$; see Figure~\ref{figure:setsPi}.

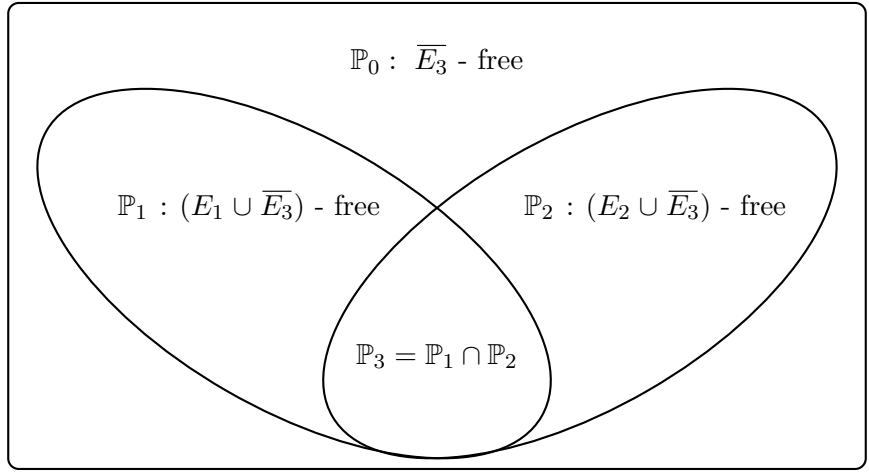
\begin{figure}[htbp]
\centering
\begin{tikzpicture}[thick]
  \draw[rounded corners=4pt] (-2.2,-2.6) rectangle (9.2,3.6);

  \node[anchor=north] at (3.5,3.2) {$\Pe_{0}:~\overbar{E_{3}}\text{ - free}$};

  \begin{scope}[shift={(1.6,0)},rotate=-30]   \draw (0,0) ellipse (3.8cm and 1.8cm); \end{scope}
  \begin{scope}[shift={(5.4,0)},rotate=30]  \draw (0,0) ellipse (3.8cm and 1.8cm); \end{scope}

  \node[align=center,text width=4.2cm] at (1,0.9)
        {$\Pe_{1}:(E_{1}\cup \overbar{E_{3}})\text{ - free}$};
  \node[align=center,text width=4.2cm] at (6.4,0.9)
        {$\Pe_{2}:(E_{2}\cup \overbar{E_{3}})\text{ - free}$};

  \node at (3.5,-1.1) {$\Pe_{3}=\Pe_{1}\cap \Pe_{2}$};
\end{tikzpicture}

%
%
\caption{The sets $\Pe_i \subseteq \Pe(G)$ with $i \in \{0,1,2,3\}$.}
\label{figure:setsPi}      
\end{figure}

With this notation Dirac's Corollary 2 (Theorem~\ref{theorem:diraccorollary2}) is the statement that $|\Pe_{0}| - |\Pe_{3}|$ is even, and hence that $|\Pe_{0}| + |\Pe_{3}|$ is even. Berge's Stronger Theorem (Theorem~\ref{theorem:bergestrong}) is that $|\Pe_{1}| + |\Pe_{2}|$ is even (note that the number of permutations without elements from $E_{2} \cup E_{3}$ equals $|\Pe_{2}|$ by reversing the orders of the permutations). We do not see an easy argument getting from one of these two results to the other. But note that if $G$ is complete (that is $E_{1} = \ems$) then $|\Pe_{0}| = |\Pe_{1}|$ and $|\Pe_{2}| = |\Pe_{3}|$, and Dirac's result and Berge's result are the same. Similarly, when $E_{2} = \ems$.

The above results of Dirac and Berge may be formulated together as follows. 

\begin{theorem} \head{(Berge-Dirac Theorem)}
Let $G$ be a mixed graph with $n \geq 2$ vertices. Let $\Pe_{0}$ denote the set of permutations of $G$ without elements from $\overbar{E_{3}}$, let $\Pe_{i}$ denote the set of permutations of $G$ without elements from $E_{i} \cup \overbar{E_{3}}$ $(i\in \{1,2\})$, and let $\Pe_{3} = \Pe_{1} \cap \Pe_{2}$. Then, either $\Pe_0$, $\Pe_1$, $\Pe_2$, $\Pe_3$ are all even; or they are all odd;
or $\Pe_0$, $\Pe_3$ are both even and $\Pe_1$, $\Pe_2$ are both odd; or $\Pe_0$, $\Pe_3$ are both odd and $\Pe_1$, $\Pe_2$ are both even. Each of the four cases occurs for infinitely many mixed graphs. In particular, $|\Pe_{0}| + |\Pe_{1}| + |\Pe_{2}| + |\Pe_{3}|$  is even.
\end{theorem}

We do not know any simpler proof of the Berge-Dirac Theorem than the above combination. Adding a new vertex to a mixed graph $G$ from one of the four cases and adding edges directed from the new vertex to all vertices of $G$ results in a new mixed graph from the same case. Thus it is enough to exhibit just one mixed graph for each case. An undirected edge is a graph of case 1; a directed edge is a graph of case 2; a graph on three vertices with one undirected edge, one directed edge and one missing edge is a graph of case 3; finally a graph on five vertices consisting of a directed 5-cycle and three undirected cords forming a path of length 3 and two missing cords from the same vertex is a graph of case 4. 

Note that in the Berge-Dirac formula we may replace any + by a -, for example
\begin{center}
   $|\Pe_{0}| - |\Pe_{1}| - |\Pe_{2}| + |\Pe_{3}|$ is always even
\end{center}
which means that the number of permutations containing at least one element from $E_{1}$, at least one element from $E_{2}$, and no element from $\overbar{E_{3}}$ is even.

\end{document}